\documentclass[10pt,twoside]{article}
\usepackage[utf8]{inputenc}
\usepackage[english]{babel}
\usepackage{fancyhdr}
\usepackage{amssymb,amsmath,amsthm}
\usepackage{blindtext}
\usepackage{hyperref}
\usepackage{caption}
\usepackage[square,numbers]{natbib}

\usepackage{fancyhdr}
\usepackage{amssymb,amsmath} %,amsthm}
\usepackage{mathtools}
\usepackage{hyperref}
\usepackage{algorithm}
\usepackage{mathrsfs}
\usepackage[square,numbers]{natbib}
\usepackage{algpseudocode}

\def\R{\mathbb R}
\def\N{\mathbb N}
\def\Astar{\mathscr Z}
\def\alphamax{\alpha_{\text{max}}}
\def\kA{\bar k}
\def\Lmax{L^{\text{max}}}
\def\LS{\ensuremath{\mathcal L^0}}

\DeclarePairedDelimiter{\norm}{\lVert}{\rVert}
\DeclarePairedDelimiter{\abs}{\lvert}{\rvert}
\DeclareMathOperator{\sign}{sign}

\DeclareMathOperator{\cone}{cone}
\DeclareMathOperator{\bigO}{\mathcal O}

\newtheorem{assumption}{Assumption}
\newtheorem{definition}{Definition}
\newtheorem{lemma}{Lemma}
\newtheorem{proposition}{Proposition}
\newtheorem{theorem}{Theorem}
\newtheorem{corollary}{Corollary}
\newtheorem{remark}{Remark}

\hypersetup{
    colorlinks=false,
    urlcolor=black,
    pdfborder={0 0 0}
}

\newcommand{\email}[1]{E-mail: \href{mailto:#1}{\texttt{#1}}}

\begin{document}
\thispagestyle{plain}

\setcounter{page}{1}

{\centering
%-----------------------%
% INSERT HERE THE TITLE %
%-----------------------%
{\Large \bfseries Complexity results and active-set identification of a derivative-free method for bound-constrained problems}

\bigskip\bigskip
%-------------------------%
% INSERT HERE THE AUTHORS %
%-------------------------%
Andrea Brilli$^*$, Andrea Cristofari$^\dag$, Giampaolo Liuzzi$^*$, Stefano Lucidi$^*$
\bigskip

}

%----------------------------------%
% INSERT HERE AUTHORS' INFORMATION %
%----------------------------------%
\begin{center}
\small{\noindent$^*$Department of Computer, Control and Management Engineering \\
Sapienza University of Rome, Italy \\
Via Ariosto, 25, 00185 Rome, Italy \\
\email{brilli@diag.uniroma1.it, liuzzi@diag.uniroma1.it, lucidi@diag.uniroma1.it} \\
\bigskip
$^\dag$Department of Civil Engineering and Computer Science Engineering \\
University of Rome ``Tor Vergata'' \\
Via del Politecnico, 1, 00133 Rome, Italy \\
\email{andrea.cristofari@uniroma2.it}  \\
}
\end{center}

\bigskip\par\bigskip\par
\noindent \textbf{Abstract.}
In this paper, we analyze a derivative-free line search method designed for bound-constrained problems. Our analysis demonstrates that this method exhibits a worst-case complexity comparable to other derivative-free methods for unconstrained and linearly constrained problems. In particular, when minimizing a function with $n$ variables, we prove that at most $\bigO{(n\epsilon^{-2})}$ iterations are needed to drive a criticality measure below a predefined threshold $\epsilon$, requiring at most $\bigO{(n^2\epsilon^{-2})}$ function evaluations.
We also show that the total number of iterations where the criticality measure is not below $\epsilon$ is upper bounded by $\bigO{(n^2\epsilon^{-2})}$.
Moreover, we investigate the method capability to identify active constraints at the final solutions.
We show that, after a finite number of iterations, all the active constraints satisfying the strict complementarity condition are correctly identified.


\bigskip\par
\noindent \textbf{MSC2000 subject classifications.} 90C56. 90C30. 90C26.

\section{Introduction}
Let us consider the following nonlinear bound-constrained optimization problem:
\begin{equation}\label{prob}
\begin{split}
& \min\ f(x) \\
& \ \ l_i \le x_i \le u_i,\qquad i=1,\dots,n,
\end{split}
\end{equation}
where $l_i,u_i\in\R\cup\{\pm\infty\}$, $l_i<u_i$, $i=1,\dots,n$.

We restrict ourselves to considering problem~\eqref{prob} when function values are given by a time consuming black-box oracle. Hence, the analytical expression of $f$ is not available and first-order information cannot be explicitly used nor approximated within a reasonable amount of time. In such a context, derivative-free methods~\cite{audet2017derivative,conn2009introduction,larson_menickelly_wild_2019} are usually employed to solve the problem.

In the literature, several derivative-free methods have been proposed to solve problem~\eqref{prob} (even with more general constraints).
In particular, we can distinguish among model-based methods~\cite{conejo2013global,conn2009introduction,gratton2011active,gumma2014derivative,hough2022model,powell2015fast}, where the objective function is sampled in a neighborhood of the current point to build an appropriate model to be minimized, direct-search methods~\cite{audet2006mesh,gratton2019direct,kolda:03,kolda2007stationarity,lewis1999pattern}, where the objective function is sampled in a neighborhood of the current point in order to find descent, and line search methods~\cite{lucidi2002derivative,lucidi2002objective}, where directions are explored by allowing the stepsize to dynamically expand.

For model-based and direct-search methods applied to problems with linear constraints (thus including~\eqref{prob}), a worst-case analysis can be found in~\cite{gratton2019direct,hough2022model}, providing upper bounds on the maximum number of iterations and function evaluations needed to drive a criticality measure below a prespecified threshold.
In particular, in~\cite{gratton2019direct}, it is shown that at most $\bigO{(n \epsilon^{-2})}$ iterations and $\bigO{(n^2 \epsilon^{-2})}$ function evaluations are needed, for a (deterministic) direct-search method, to produce the first point with a criticality measure below $\epsilon>0$, matching the same complexity for the unconstrained case~\cite{dodangeh2016optimal,vicente2013worst}.
In~\cite{hough2022model}, similar bounds of $\bigO{(k_D^2 \epsilon^{-2})}$ iterations and $\bigO{(n k_D^2 \epsilon^{-2})}$ function evaluations are obtained, matching the same complexity for the unconstrained case~\cite{garmanjani2016trust}, with $k_D$ being a problem dimension-dependent constant which define a fully linear model.
Methods based on finite-difference gradient approximations may be used in this context. Such methods usually guarantee favorable worst-case evaluation complexity (e.g., $\bigO{(n \epsilon^{-2})}$ in the unconstrained case~\cite{grapiglia:2023}), but they can be inefficient in the presence of noise, unless carefully implemented~\cite{conn2009introduction,shi:2023}. Furthermore, concerning the complexity bound it can be noted that the constant defining the $\bigO{(\epsilon^{-2})}$ bound in \cite{grapiglia:2023} depends quadratically on the Lipschitz constant of the gradient whereas a linear dependence is typical for derivative-free methods. As discussed in \cite{vicente:13}, the Lipschitz constant of the gradient of the objective function can in turn depend exponentially on the dimension of the problem.  

In the current paper, we analyze a line search method to solve problem~\eqref{prob}. The algorithm under analysis is a modification of the one proposed in~\cite{lucidi2002global}, equipped with a line search technique described in~\cite{brilli2024worst}. In particular, for any considered direction (i.e., a vector of the canonical basis), the line search technique first checks for a sufficient decrease in the objective function using a given stepsize. Then, if such a decrease is obtained, an extrapolation (or expansion) phase starts, where increasingly larger values of the stespize are tried until some conditions are met.
This approach allows us to obtain complexity and identification results that extend those existing in the literature for direct-search~\cite{gratton2019direct} and model-based~\cite{hough2022model} methods.
In particular, let us summarize the main contributions of this paper below.

\begin{itemize}
\item The first contribution of the current paper is providing a worst-case analysis for the proposed line search method, which yields to the same bounds for direct search~\cite{gratton2019direct}, that is, $\bigO{(n \epsilon^{-2})}$ iterations and $\bigO{(n^2 \epsilon^{-2})}$ function evaluations to produce the first point with a criticality measure below $\epsilon>0$.
Additionally, for the proposed algorithm, we are able to bound the \textit{total} number of iterations where the criticality measure is not below $\epsilon$, thus going beyond the complexity results for direct-search methods given in~\cite{gratton2019direct}.

\item  The second contribution of the current paper is to show finite identification of the active constraints for the proposed line search method.
Such a property is usually desirable for an optimization algorithm
due to, among other things, the possibility of saving function evaluations if one recognizes the surfaces where a stationary points lies.
Furthermore, in several applications, we might be interested only in the identification of the surface containing an optimal solution (or its support).
In the literature, finite active-set identification was established
for smooth optimization algorithms and proximal methods (see, e.g.,~\cite{bertsekas1976goldstein,bomze2019first,burke1988identification,hare2004identifying,wright1993identifiable}), also providing complexity bounds in some cases~\cite{bomze2020active,cristofari2022active,nutini2019active}.
Moreover, many active-set approaches were designed for derivative-based optimization with bound constraints (see, e.g.,~\cite{birgin:2002,cristofari:2017,desantis:2012,hager:2006,yuan:2011,yuan:2015}).

In a derivative-free setting, parameter-dependent estimates were used in~\cite{gratton2011active,lewis2010active,lucidi2002objective}, allowing for finite identification of active constraints if certain conditions hold.
Moreover, finite identification results have been shown in~\cite{cristofari2021derivative} for a method using an inner approximation approach to minimize a function over the convex hull of a given set of vectors, meaning that, in finite time, the algorithm is able to identify, under appropriate assumptions, the vectors with zero weight in the convex combination representing the final solution.
Also note that even though the proposed algorithm, as highlighted above, gives the same complexity bound as the direct-search scheme analyzed in~\cite{gratton2019direct}, the latter does not guarantee finite active-set identification as the iterates may get very close to the boundary but never lie at the boundary.

Here, we show that the proposed algorithm correctly identifies the active constraints satisfying the strict complementarity condition in a finite number of iterations, without using any parameter-dependent estimate. Namely, this feature is just an intrinsic property of the proposed algorithm.
Such a result is obtained by using some tools from the analysis of derivative-based methods ~\cite{cristofari2022active,nutini2019active}. More specifically, we define a measure which represents the minimum strict complementarity among the active constraints, thereby providing a neighborhood of the limit points where the active-set identification holds.
Let us remark that also this identification result is obtained thanks to the extrapolation technique used in the line search procedure, which allows the stepsize to expand until we hit the border of the feasible set when we are in a neighborhood of a stationary point.
\end{itemize}

\subsection{Notations}
Given $l,u\in\R^n$, we denote $[l,u] = \{x\in\R^n:\ l_i\leq x_i\leq u_i,\ i=1,\dots,n\}$ the feasible set of problem \eqref{prob}. If ${\cal A} = \{a_1,\dots,a_p\}\subset\R^n$ is a (finite) set of vectors, we denote
\[
\cone({\cal A}) = \left\{x\in\R^n:\ x=\sum_{i=1}^p \beta_ia_i,\ \beta_i\geq 0,\ i=1,\dots,p\right\}.
\]
Given $v\in \R^n$ and a set $S\subseteq\R^n$, we denote by $v_S$ the projection of $v$ onto $S$. Given $\beta\in\R$, we indicate the sign of $\beta$ by $\sign(\beta)$, that is, $\sign(\beta)$ is $-1$ if $\beta<0$, $0$ if $\beta = 0$ and $1$ if $\beta > 0$.
Finally, $\|v\|$ denotes the $\ell_2$-norm of vector $v$.

The paper is organized as follows. In Section~\ref{sec2}, we define a line search algorithm for the solution of problem~\eqref{prob}. Section~\ref{sec3} is devoted to the analysis of the asymptotic convergence of the proposed algorithm, followed by the derivation of worst-case complexity bounds. In
Section~\ref{sec4}, we show finite active-set identification of the proposed algorithm. Finally, in Section~\ref{sec5}, we draw some conclusions.

\section{The algorithm}\label{sec2}
This section is concerned with the definition of a line search algorithm to solve problem \eqref{prob}.
The proposed method, denoted as Algorithm~\ref{alg:dfl}, is inspired by the method proposed in~\cite{lucidi2002derivative}
and uses some ideas from the one proposed in~\cite{brilli2024worst} for unconstrained problems.

At each iteration $k$, starting from the current iterate $x_k$, the algorithm sets $y_k^1 = x_k$ and explores the coordinate directions $\pm e_i$, $i=1,\ldots,n$, using stepsizes
$\nu_k^i = \max\{\tilde\alpha_k^i,c\Delta_k\}$, where
\[
\Delta_k = \max_{i=1,\dots,n}\{\tilde\alpha_k^i\},
\]
and the quantities $\tilde\alpha_k^i$, $i=1,\dots,n$, are \textit{tentative stepsizes}
updated throughout the iterations.
Then, the scheme produces \textit{actual stepsizes} $\alpha_k^i$ to compute intermediate points $y_k^{i+1} = y_k^i + \alpha_k^i d_k^i$, with $d_k^i \in \{\pm e_i\}$, $i=1,\dots,n$.
In particular, for any $y_k^i$, if $y_k^i+\nu_k^i d_k^i$ is infeasible (i.e., $y_k^i$ is too close to a bound) or does not provide a sufficient decrease in the objective function, then we set $\alpha_k^i = 0$ (i.e., $y_k^{i+1} = y_k^i$).
Otherwise,
a sufficient decrease in the objective function is achieved by moving along $d_k^i$ with
a feasible stepsize $\alpha_k^i$ determined by a line search procedure which will be described later.
Then, we set $x_{k+1} = y_k^{n+1}$ and prepare for the next iteration $k+1$.

As a final note regarding our proposed scheme, we refer to $k$ as a \textit{successful iteration} if $x_{k+1} \ne x_k$, indicating that at least one positive stepsize $\alpha_k^i$, $i = 1,\ldots,n$, has been computed. Conversely, we refer to $k$ as an \textit{unsuccessful iteration} if $x_{k+1} = x_k$, that is, if $\alpha_k^i = 0$ for all $i = 1,\ldots,n$.
Depending on whether an iteration $k$ is successful or not, we use specific rules to update the tentative stepsizes for the next iteration $k+1$.
In more detail, for a successful iteration $k$, each $\tilde\alpha_{k+1}^i$ is set to $\alpha_k^i$ if the latter is positive, whereas $\tilde\alpha_{k+1}^i$ is set to $\nu_k^i$ otherwise.
For an unsuccessful iteration $k$, each $\tilde\alpha_{k+1}^i$ is set to
$\theta\nu_k^i$, with $\theta \in (0,1)$.

\begin{algorithm}[ht]
\caption{Derivative-free line search algorithm}\label{alg:dfl}
\begin{algorithmic}[1]
\State{\bf given} $x_0\in[l,u]$, $\theta\in(0,1)$, $\delta\in(0,1)$, $\gamma > 0$, $c\in(0,1]$, $\tilde\alpha_0^i > 0$,  $i=1,\dots,n$
\For{$k=0,1,\dots$}
\State set $\Delta_k = \max_{i=1,\dots,n}\{\tilde\alpha_k^i\}$
\State set $y_k^1 = x_k$
\For{$i=1,\dots,n$}
\State set $\nu_k^i = \max\{\tilde\alpha_k^i,c\Delta_k\}$
\State compute $d_k^i$ and $\alpha_k^i$ by the {\tt line search}$(y_k^i,i,\gamma,\delta,\nu_k^i)$
\State set $y_k^{i+1} = y_k^i + \alpha_k^id_k^{i}$
\EndFor
\State set $x_{k+1} = y_k^{n+1}$
\If{$x_{k+1} \neq x_k$}
\State set $\tilde\alpha_{k+1}^i = \begin{cases} \alpha_k^i \quad \text{if } \alpha_k^i > 0 \\
        \nu_k^i \quad \text{otherwise} \end{cases} \quad i = 1,\ldots,n$
\Else
\State set $\tilde\alpha_{k+1}^i = \theta\nu_k^i$, $i=1,\dots,n$
\EndIf
\EndFor
\end{algorithmic}
\end{algorithm}

Given a feasible point $x$,
the exploration of the $i$th coordinate direction $e_i$ is performed by a line search procedure outlined in Algorithm~\ref{alg:ls}. First, we check if the given stepsize $\nu$ is feasible along $\pm e_i$, that is, if either $x + \nu e_i$ or $x - \nu e_i$ is feasible. If this is not the case, then we quit the line search returning a zero step length to indicate a failure.
Otherwise, we try to determine if one between $e_i$ and $-e_i$ is a ``good'' descent direction, that is, if a sufficient decrease in the objective function can be obtained by using a feasible stepsize. If neither $e_i$ nor $-e_i$ qualifies as a suitable descent direction, then the line search procedure terminates, returning a zero step length to indicate a failure.
Conversely, if a sufficient decrease of $f$ is obtained, then
an \textit{extrapolation} (or \textit{expansion}) phase starts (i.e., lines 12--15), where we try to increase the stepsize to the maximum extent while preserving feasibility and guaranteeing the sufficient decrease condition. Specifically, the \textit{while} loop keeps expanding the stepsize as long as the most recently accepted point remains strictly within the bounds (i.e., $\alpha < \alphamax$) and the new tentative point is sufficiently better than the last accepted one (i.e., $f(x + \omega d) \le f(x+\alpha d) - \gamma (\omega-\alpha)^2$).

Note that the proposed method is an adaptation of standard line search schemes with a few technical modifications introduced solely to obtain the desired complexity bounds. While numerical results illustrating the performance of line search variants are available in the literature~\cite{brilli2024worst,lucidi2002derivative}, here we focus on highlighting new theoretical properties about worst-case complexity and finite active-set identification of such a scheme.

\begin{algorithm}[ht]
\caption{line search($x$,$i$,$\gamma$,$\delta$,$\nu$)}
\label{alg:ls}
\begin{algorithmic}[1]
\If{$\nu > \max\{u_i-x_i,x_i-\ell_i\}$}
{return $d=e_i$, $\alpha = 0$}
\EndIf
\State set $\bar\alpha =\nu$
\If {$\bar{\alpha} \leq x_i - l_i$ and $f(x - \bar\alpha e_i) \le f(x) - \gamma \bar\alpha^2$}
\State {set $d = -e_i$, $\alphamax = x_i - l_i$ and go to line 12}
\EndIf
\If {$\bar{\alpha} \leq u_i - x_i$ and $f(x + \bar\alpha e_i) \le f(x) - \gamma \bar\alpha^2$}
\State set $d = e_i$, $\alphamax = u_i - x_i$ and go to line 12
\Else
\State return $d$ and $\alpha=0$
\EndIf
\State set $\alpha=\bar\alpha$ and $\omega = \min\{\alpha/\delta,\alphamax\}$
\While {\big($\alpha < \alphamax$ and $f(x + \omega d) \le f(x+\alpha d) - \gamma (\omega-\alpha)^2$\big)}\label{LS_step12}
\State set $\alpha = \omega$ and $\omega = \min\{\alpha/\delta,\alphamax\}$
\EndWhile\label{LS_step14}
\State {\bf return} $d$ and $\alpha$
\end{algorithmic}
\end{algorithm}

\section{Convergence and worst-case complexity}\label{sec3}
This section is devoted to the theoretical analysis of Algorithm~\ref{alg:dfl}.

\subsection{Assumptions and preliminary results}
First, let us define the following level set:
\[
\LS = \{x \in \R^n \colon f(x) \le f(x_0)\}.
\]

Throughout the paper, the following assumptions will be always considered satisfied, even if not explicitly invoked.

\begin{assumption}\label{assump1}
There exists an open convex set $\mathcal S \supseteq \LS$ such that the
objective function $f \colon \R^n \to \R$ is continuously differentiable with a Lipschitz continuous gradient $\nabla f$ with constant $L>0$ over $\mathcal S$, i.e.,
\[
\|\nabla f(x)-\nabla f(y)\| \leq L\|x-y\|\quad\forall x,y\in \mathcal S.
\]
Moreover, $f$ is bounded from below over $[l,u]$, i.e., a constant $f_{\min}\in\R$ exists such that
\[
f_{\min} \leq f(x)\quad\forall x\in[l,u].
\]
\end{assumption}
\begin{assumption}\label{assump2}
A constant $M_g\geq0$ exists such that
\[
\|\nabla f(x)\| \leq M_g
\]
for all $x\in \LS$.
\end{assumption}

Under Assumption~\ref{assump1}, we can also provide the following definition of coordinate-wise Lipschitz constants.
\begin{definition}\label{coord_lip}
The coordinate-wise Lipschitz constants $L_i>0$, $i=1,\ldots,n$, of $\nabla f$ are such that, for all $x\in\mathcal S$,
\[
\abs{\nabla_i f(x + se_i) - \nabla_i f(x)} \le L_i \abs{s} \quad \forall s\in\R\colon x+se_i \in [l,u], \quad i=1,\ldots,n,
\]
where the set $\mathcal S$ is defined as in Assumption~\ref{assump1}.
Moreover,
\begin{equation}\label{lmax}
\Lmax = \max_{i=1,\ldots,n} L_i.
\end{equation}
\end{definition}

Now, given $x\in[l,u]$,
let us introduce the following criticality measure:
\[
\chi(x) = \max_{\substack{x+d\in[l,u] \\ \|d\|\leq 1}} - \nabla f(x)^\top d.
\]
The above measure has been successfully used
in the analysis of some direct search methods for linearly constrained problems~\cite{gratton2019direct,kolda:03,kolda2007stationarity}.
It can be interpreted as the progress on a first-order model in a ball centered at $x$ with unit radius subject to feasibility constraints~\cite{conn2000trust,kolda2007stationarity}, thus generalizing $\|\nabla f(x)\|$ from the unconstrained setting.
Originally proposed in~\cite{conn1996convergence} for more general constraints and further analyzed in~\cite{conn2000trust}, $\chi(x)$ is continuous, non-negative and such that $\chi(x)=0$ if and only if $x$ is a KKT point. So, we can define a stationary point as follows.

\begin{definition}
A point $x^* \in [l,u]$ is said to be a stationary point of problem~\eqref{prob} if $\chi(x^*)=0$.
\end{definition}

Next, given $\epsilon \ge 0$, we define the set of $\epsilon$-active constraints at $x\in[l,u]$ as
\begin{align*}
I_l(x,\epsilon) & = \{i: x_i \leq l_i+\epsilon\},\\
I_u(x,\epsilon) & = \{i: x_i \geq u_i-\epsilon\}.
\end{align*}
Namely, $I_l(x,\epsilon)$ and $I_u(x,\epsilon)$ denote the sets of lower and upper bound constraints, respectively, that are nearly active at $x$ with a tolerance $\epsilon$.
Accordingly, let us define $N(x,\epsilon)$ as the $\epsilon$-normal cone generated by the $\epsilon$-active constraints, that is,
\[
N(x,\epsilon) = \cone\Big({\{-e_i, i\in I_l(x,\epsilon)\}\cup\{e_i, i\in I_u(x,\epsilon)\}\cup\{0\}\Big)},
\]
while the $\epsilon$-tangent cone $T(x,\epsilon)$ is the polar of $N(x,\epsilon)$, that is,
\[
T(x,\epsilon) = N(x,\epsilon)^\circ = \{d\in\R^n: d^\top v \leq 0,\ \forall\ v\in N(x,\epsilon)\}.
\]
The use of $\epsilon$-normal and $\epsilon$-tangent cones is a well known tool in the analysis of direct search methods applied to linearly constrained problems~\cite{gratton2019direct,kolda:03,kolda2007stationarity}. Essentially, the set $x + T(x,\epsilon)$ is an approximation of the feasible region near a feasible point $x$, that is,
moving from $x$ along any direction in $T(x,\epsilon)$ with a stepsize less than or equal to $\epsilon$ ensures that all constraints stay satisfied.

In our case, considering the structure of the feasible set of problem~\eqref{prob}, it is straightforward to verify that
a set of generators for $T(x,\epsilon)$ is given by
\begin{equation}\label{genTk}
G_{T(x,\epsilon)} =
\{-e_i, i\not\in I_l(x,\epsilon)\}\cup\{e_i, i\not\in I_u(x,\epsilon)\}\cup\{0\},
\end{equation}
that is, $T(x,\epsilon) = \cone{(G_{T(x,\epsilon)})}$.

The following two propositions from~\cite{kolda2007stationarity} show how $\chi(x_k)$ can be upper bounded
by using of the projection of $-\nabla f(x_k)$ onto $T(x_k,\epsilon)$ and $N(x_k,\epsilon)$.

\begin{proposition}[{\cite[Proposition 8.2]{kolda:03}}]\label{prop:chi_ub}
If $x\in[l,u]$, then for all $\epsilon\geq 0$ we have that
\[
\chi(x) \leq \|(-\nabla f(x))_{T(x,\epsilon)}\| + \epsilon\sqrt{n}\|(-\nabla f(x))_{N(x,\epsilon)}\|.
\]
\end{proposition}

\begin{proposition}[{\cite[Proposition 8.1]{kolda:03}}]\label{prop:gd}
Given $\epsilon \ge 0$, let $G_{T(x,\epsilon)}$ be defined as in~\eqref{genTk}. If $(-\nabla f(x))_{T(x,\epsilon)}\neq 0$, then there exists $d\in G_{T(x,\epsilon)}$ such that
\[
\dfrac 1{\sqrt{n}}\|(-\nabla f(x))_{T(x,\epsilon)}\| \leq -\nabla f(x)^\top d.
\]
\end{proposition}

In the convergence analysis of Algorithm~\ref{alg:dfl}, Propositions~\ref{prop:chi_ub}--\ref{prop:gd} will allow us to relate $\chi(x_k)$ with $\Delta_{k+1}$ for every iteration $k$ (see Theorem~\ref{th:chi} below).
In particular, this will be obtained by applying the above results with $\epsilon = \Delta_k$ and using the following relation between $\Delta_k$ and $\Delta_{k+1}$.

\begin{lemma}\label{lemmalegame}
Let $\{\Delta_k\}$ be the sequence of maximum tentative stepsizes produced by Algorithm~\ref{alg:dfl}. Then, for all $k$ we have
\[
\Delta_{k+1}
\begin{cases}
\geq \Delta_k \quad & \text{if } x_{k+1} \neq x_k \\
= \theta\Delta_k \quad & \text{if } x_{k+1} = x_k.
\end{cases}
\]
Therefore, $\Delta_k\leq \Delta_{k+1}/\theta$ for all $k$.
\end{lemma}
\begin{proof}
First, let us consider the case where $x_{k+1} \ne x_k$ (i.e., $k$ is a successful iteration). In this case, the algorithm sets $\tilde\alpha_{k+1}^i = \alpha_k^i$ if $\alpha_k^i > 0$ and $\tilde\alpha_{k+1}^i = \nu_k^i$ if $\alpha_k^i = 0$. Since, from the line search procedure, $\alpha_k^i \geq\nu_k^i$ for all $i = 1,\ldots,n$, we can write
\[
\tilde\alpha_{k+1}^i \geq \nu_k^i \geq \tilde\alpha_k^i, \quad  i = 1,\ldots,n,
\]
where the last inequality follows from the definition of $\nu_k^i$. Then, using the definition of $\Delta_k$, we get $\Delta_{k+1}\geq \Delta_k$.

Now, let us consider the case where $x_{k+1} = x_k$ (i.e., $k$ is an unsuccessful iteration). In this case, the algorithm sets $\tilde\alpha_{k+1}^i = \theta\nu_k^i = \theta\max\{\tilde\alpha_k^i,c\Delta_k\}$. Namely, for all $i=1,\dots,n$, we have that
\begin{equation}\label{tildealfadef}
\tilde\alpha_{k+1}^i =
\begin{cases}
    \theta\tilde\alpha_k^i \quad & \text{if } \tilde\alpha_k^i \geq c\Delta_k, \\
    \theta c\Delta_k \quad & \text{if } \tilde\alpha_k^i < c\Delta_k.
\end{cases}
\end{equation}
From the definition of $\Delta_k$ and the fact that $c\in(0,1]$, it follows that
$\tilde\alpha_{k+1}^i \leq \theta\Delta_k$ for all $i = 1,\ldots,n$,
implying that
\begin{equation}\label{Delta_next}
\Delta_{k+1} = \max_{i=1,\ldots,n}{\tilde \alpha^i_{k+1}} \le \theta\Delta_k.
\end{equation}
Now, let $\bar\imath\in\{1,\dots,n\}$ be such that $\tilde\alpha_k^{\bar\imath} = \Delta_k$.
Since $c\in(0,1]$, we have $\tilde\alpha_k^{\bar\imath} \ge c\Delta_k$ and then, recalling~\eqref{tildealfadef}, we get
$\tilde\alpha_{k+1}^{\bar\imath} = \theta\tilde\alpha_k^{\bar\imath} = \theta\Delta_k$.
It follows from~\eqref{Delta_next} that
\[
\Delta_{k+1} = \tilde\alpha_{k+1}^{\bar\imath} = \theta\Delta_k,
\]
which concludes the proof.
\end{proof}

\subsection{Global convergence}

In the following proposition, we show that Algorithm \ref{alg:dfl} and Algorithm \ref{alg:ls} are well defined, i.e., Algorithm \ref{alg:ls} cannot cycle so that Algorithm \ref{alg:dfl} produces infinite sequences of points and stepsizes.

\begin{proposition}\label{th:conv1_LS}
Algorithm \ref{alg:dfl} is well defined, i.e., it produces infinite sequences $\{x_k\}$, $\{\alpha_k^i\}$, $\{\tilde\alpha_k^i\}$, $i=1,\dots,n$. \end{proposition}

\begin{proof} To prove that Algorithm \ref{alg:dfl} is well defined, we have to show that the line search procedure cannot infinitely cycle over steps \ref{LS_step12}--\ref{LS_step14} in the while loop.
Let us suppose, by contradiction, that the while loop does not terminate, i.e., we always have $\alpha < \alphamax$. At the $j$-th iteration of the while loop, we have
\[
\alpha = \frac{\bar\alpha}{\delta^j}
\]
with $\delta < 1$. If $\alphamax$ is finite, then we have $\alpha > \alphamax$, for $j$ sufficiently large, which is a contradiction. Otherwise, if $\alphamax = +\infty$, then, for every $j$ we have
\begin{equation}\label{condwhile}
f(x + \omega d) \le f(x+\alpha d) - \gamma (\omega-\alpha)^2,
\end{equation}
with $\omega = \alpha/\delta = \bar\alpha/\delta^{j+1}$.
Now, for $j$ sufficiently large, \eqref{condwhile} contradicts Assumption~\ref{assump1}, i.e., that $f$ is bounded from below on the feasible set.
\end{proof}

In the following, we present some results concerning the global convergence of Algorithm \ref{alg:dfl} to stationary points.
More specifically, by extending some results from~\cite{brilli2024worst,lucidi2002derivative}, we establish a relationship between $\nabla f(x_k)$, with specific directions $d$, and the largest tentative step length at $\Delta_{k+1}$.
We will show that the bound depends, besides on the problem dimension $n$, on the Lipschitz constant $L$ and the algorithm parameters $\gamma$, $\theta$ and $\delta$.

From now on, let us denote
\begin{equation}\label{Tk_Gk}
T_k := T(x_k,\Delta_k) \quad \text{and} \quad G_k := G(x_k,\Delta_k),
\end{equation}
where $G_k$ is the set of generators of $T_k$ and is defined as in~\eqref{genTk}.

\begin{theorem}\label{bound_grad}
Let $\{x_k\}$ be the sequence produced by Algorithm~\ref{alg:dfl}.
Then, for all $k$ and for all $d\in G_k$, we have that
\begin{equation}\label{grad_app}
-\nabla f(x_k)^\top d \le
\begin{cases}
\left(\dfrac{\gamma+L}{\delta}+L\sqrt{n}\right) \Delta_{k+1} \quad & \text{if } x_{k+1} \neq x_k, \\[1.2em]
\left(\dfrac{\gamma+\Lmax}{\theta}\right)\Delta_{k+1} \quad & \text{if } x_{k+1} = x_k.
\end{cases}
\end{equation}
\end{theorem}

\begin{proof}
For all $k$, the result trivially holds for $d
 =0$. Now, consider an iteration $k$ such that $x_{k+1}\neq x_k$ (i.e., a successful iteration).
The following cases can occur, recalling that the analysis is limited to considering directions $\pm e_i$, $i=1,\ldots,n$, belonging to $G_k$.
\begin{itemize}
\item[(1a)] $e_i\in G_k$, $(y_k^{i+1})_i = l_i$ and $(y_k^i)_i = l_i$. Then, $\alpha_k^i=0$ and $\tilde\alpha_{k+1}^i=\nu_k^i$.
From the instructions of the line search procedure, we have that
\[
f(y_k^i +\tilde\alpha_{k+1}^ie_i) > f(y_k^i) - \gamma(\tilde\alpha_{k+1}^i)^2.
\]
By the mean value theorem, we have
\[
f(y_k^i + \tilde\alpha_{k+1}^ie_i)-f(y_k^i)= \tilde\alpha_{k+1}^i\nabla_i f(\xi_k^i),
\]
where $\xi_k^i = y_k^i + t_k^i\tilde\alpha_{k+1}^ie_i$ and $t_k^i\in(0,1)$. Then,
\[
-\nabla_i f(\xi_k^i) < \gamma {\tilde\alpha_{k+1}^i}.
\]
It follows that
\[
-\nabla_i f(\xi_k^i)+\nabla_i f(x_k)-\nabla_i f(x_k) <  \gamma {\tilde\alpha_{k+1}^i}
\]
and we can write
\begin{equation}\label{prop:baoundgr0}
\begin{split}
-\nabla_i f(x_k) & < \gamma{\tilde\alpha_{k+1}^i} -(\nabla_i f(x_k)-\nabla_i f(\xi_k^i)) \\
& \le
\gamma{\tilde\alpha_{k+1}^i} + L\|x_k-\xi_k^i\| \\
& \le \gamma{\tilde\alpha_{k+1}^i} +L\|x_k-y_k^i\| + L \|y_k^i-\xi_k^i\|.
\end{split}
\end{equation}
Moreover, since $t_k^i \in (0,1)$, we have
$t^i_k \tilde\alpha_{k+1}^i \le {\tilde\alpha_{k+1}^i}$. Then,
\[
\|y_k^i-\xi_k^i\| \le {\tilde\alpha_{k+1}^i}.
\]
Hence, since $\alpha_{k+1}^i \le \tilde \alpha_{k+1}^i \le \Delta_{k+1}$ for all $i=1,\dots,n$, and taking into account that $y_k^i = x_k + \sum_{j=1}^{i-1}\alpha_k^jd_k^j$, so that $\|x_k-y_k^i\|\leq \sqrt{n}\Delta_{k+1}$, we get
\begin{equation}\label{prop:baoundgr3}
(-\nabla f(x_k))^\top e_i <
\gamma{\tilde\alpha_{k+1}^i} +L\|x_k-y_k^i\| + L{\tilde\alpha_{k+1}^i}
 \le \left({\gamma+L}+L\sqrt{n}\right) \Delta_{k+1}.
\end{equation}

\item[(1b)]{ $-e_i\in G_k$, $(y_k^{i+1})_i = l_i$ and $(y_k^i)_i>l_i$ }. Then, $\tilde\alpha_{k+1}^i=\alpha_k^i > 0$. From the instructions of the line search procedure, there exists $\beta_k^i\in \{0, \delta \tilde\alpha_{k+1}^i\}$ such that
\[
f(y_k^i -\tilde\alpha_{k+1}^ie_i) \leq f(y_k^i -\beta_k^ie_i) - \gamma(\tilde\alpha_{k+1}^i-\beta_k^i)^2 \leq f(y_k^i -\beta_k^ie_i).
\]
By the mean value theorem, we have
\[
f(y_k^i -\tilde\alpha_{k+1}^ie_i) - f(y_k^i -\beta_k^ie_i) = -(\tilde\alpha_{k+1}^i-\beta_k^i)\nabla_i f(\xi_k^i),
\]
where $\xi_k^i = y_k^i -\beta_k^ie_i+ t_k^i(\beta_k^i-\tilde\alpha_{k+1}^i)e_i$ and $t_k^i\in(0,1)$. Then,
\[
-\nabla_i f(\xi_k^i) \leq 0.
\]
Moreover, from the definition of $\beta_k^i$ and the fact that $t^i_k \in (0,1)$, we have
$(1-t^i_k) \beta_k^i \le (1-t_k^i)\tilde \alpha_{k+1}^i$. Then,
\[
\|y_k^i-\xi_k^i\| = (1-t^i_k) \beta_k^i + t_k^i \tilde \alpha^i_{k+1} \le \tilde \alpha^i_{k+1}.
\]
Hence, by similar reasoning as in case (1a), we obtain
\begin{equation}\label{prop:baoundgr4}
(-\nabla f(x_k))^\top (-e_i)
\le \left(L+L\sqrt{n}\right) \Delta_{k+1}.
\end{equation}

\item[(2a)] $-e_i\in G_k$, $(y_k^{i+1})_i = u_i$ and $(y_k^i)_i = u_i$. Then, $\alpha_k^i=0$ and $\tilde\alpha_{k+1}^i=\nu_k^i$.
Reasoning as in case (1a), with minor differences, we get
\begin{equation}\label{prop:baoundgr3_u}
(-\nabla f(x_k))^\top (-e_i) \le \left({\gamma+L}+L\sqrt{n}\right) \Delta_{k+1}.
\end{equation}

\item[(2b)]$e_i\in G_k$, $(y_k^{i+1})_i = u_i$ and $(y_k^i)_i<u_i$. Then, $\tilde\alpha_{k+1}^i=\alpha_k^i > 0$.
Reasoning as in case (1b), with minor differences, we get
\begin{equation}\label{prop:baoundgr4_u}
(-\nabla f(x_k))^\top e_i
\le \left(L+L\sqrt{n}\right) \Delta_{k+1}.
\end{equation}

\item[(3a)] $\{\pm e_i\} \cap G_k\neq\emptyset$, $l_i<(y_k^{i+1})_i<u_i$ and $y_k^i = y_k^{i+1}$. Then, $\alpha_k^i = 0$ and $\tilde\alpha_{k+1}^i = \nu_k^i$.
From the instructions of the line search procedure, we have that
\begin{equation*}
\begin{split}
f(y_k^i+\tilde\alpha_{k+1}^ie_i) & > f(y_k^i) -\gamma(\tilde\alpha_{k+1}^i)^2 \quad\text{if}\ e_i\in G_k, \\
f(y_k^i-\tilde\alpha_{k+1}^ie_i) & > f(y_k^i) -\gamma(\tilde\alpha_{k+1}^i)^2\quad\text{if}\ -e_i\in G_k.
\end{split}
\end{equation*}
By the mean value theorem, we have
\begin{equation*}
\begin{split}
f(y_k^i+\tilde\alpha_{k+1}^ie_i) - f(y_k^i) & = \nabla_i f(\xi_k^i)\tilde\alpha_{k+1}^i\quad\text{if}\ e_i\in G_k, \\
 f(y_k^i-\tilde\alpha_{k+1}^ie_i) - f(y_k^i) & = - \nabla_i f(\bar \xi_k^i)\tilde\alpha_{k+1}^i\quad\text{if}\ -e_i\in G_k.
  \end{split}
\end{equation*}
where $\xi_k^i = y_k^i + t_k^i \tilde\alpha_{k+1}^ie_i$ and $\bar \xi_k^i = y_k^i - \bar t_k^i \tilde\alpha_{k+1}^ie_i$, with $t_k^i,\bar t_k^i\in(0,1)$. Then,
\begin{equation*}
 \begin{split}
 -\nabla_i f(\xi_k^i) & < \gamma \tilde\alpha_{k+1}^i\quad\text{if}\ e_i\in G_k, \\
 \nabla_i f(\bar \xi_k^i) & < \gamma \tilde\alpha_{k+1}^i\quad\text{if}\ -e_i\in G_k.
  \end{split}
\end{equation*}
and, recalling that $t^i_k, \bar t^i_k \in (0,1)$, we also have
\begin{equation*}
\begin{split}
\|y_k^i-\xi_k^i\| & = t_k^i \tilde\alpha_{k+1}^i\leq {\tilde\alpha_{k+1}^i}\quad\text{if}\ e_i\in G_k,\\
\|y_k^i-\bar\xi_k^i\| & = t_k^i \tilde\alpha_{k+1}^i\leq {\tilde\alpha_{k+1}^i}\quad\text{if}\ -e_i\in G_k.
\end{split}
\end{equation*}
Hence, reasoning as in case (1a) (i.e., using~\eqref{prop:baoundgr0} for $e_i$ and applying minor changes to~\eqref{prop:baoundgr0}, with $\xi_k^i$ replaced by $\bar \xi_k^i$, for $-e_i$), we obtain
\begin{equation}\label{prop:baoundgr3_ul}
\begin{split}
-\nabla f(x_k)^\top e_i&\le \left({\gamma+L}+L\sqrt{n}\right) \Delta_{k+1}\quad\text{if}\ e_i\in G_k,\\
-\nabla f(x_k)^\top (-e_i)&\le \left({\gamma+L}+L\sqrt{n}\right) \Delta_{k+1}\quad\text{if}\ -e_i\in G_k.
\end{split}
\end{equation}

\item[(3b)]$d_k^i\in G_k$, $l_i<(y_{k}^{i+1})_i<u_i$ and $y_k^i\neq y_{k}^{i+1}$. Then, $\alpha_k^i > 0$, $\tilde\alpha_{k+1}^i = \alpha_k^i$.
Let us assume that $d_k^i = -e_i$, i.e., $y_k^{i+1} = y_k^i-\tilde\alpha_{k+1}^ie_i$ (the proof for the case $d_k^i = e_i$ is identical, except for minor changes). Then, from the instructions of the line search procedure, there exist
$\beta_k^i$ and $\bar \beta_k^i$ such that
$\beta_k^i \in \{0,\delta\tilde\alpha_{k+1}^i\}$, $\tilde\alpha_{k+1}^i<\bar\beta_k^i\leq\tilde\alpha_{k+1}^i/\delta$ and
\begin{equation*}
\begin{split}
    f(y_k^i-\tilde\alpha_{k+1}^ie_i) & \leq f(y_k^i-\beta_k^ie_i) -\gamma(\tilde\alpha_{k+1}^i-\beta_k^i)^2 \le f(y_k^i-\beta_k^ie_i), \\
    f\left(y_k^i-\bar\beta_k^ie_i\right) & > f(y_k^i-\tilde\alpha_{k+1}^ie_i) -\gamma\left(\bar\beta_k^i-\tilde\alpha_{k+1}^i\right)^2.
\end{split}
\end{equation*}
By the mean value theorem, we have
\begin{equation*}
\begin{split}
     f(y_k^i-\tilde\alpha_{k+1}^ie_i) - f(y_k^i-\beta_k^ie_i) & = -(\tilde\alpha_{k+1}^i-\beta_k^i)\nabla_i f(\xi_k^i),\\
    f\left(y_k^i-\bar\beta_k^ie_i\right) - f(y_k^i-\tilde\alpha_{k+1}^ie_i) & = -(\bar\beta_k^i-\tilde\alpha_{k+1}^i) \nabla_i f(\bar\xi_k^i),
\end{split}
\end{equation*}
where $\xi_k^i = y_k^i -\beta_k^ie_i- t_k^i (\tilde\alpha_{k+1}^i-\beta_k^i)e_i$ and $\bar\xi_k^i = y_k^i -\tilde\alpha_{k+1}^ie_i- \bar t_k^i (\bar\beta_k^i-\tilde\alpha_{k+1}^i)e_i$, with $t_k^i,\bar t_k^i\in(0,1)$.
Then,
\begin{equation*}
\begin{split}
     -\nabla_i f(\xi_k^i) & \le 0, \\
    \nabla_i f(\bar\xi_k^i) & < \gamma(\bar\beta_k^i-\tilde\alpha_{k+1}^i).
\end{split}
\end{equation*}
Moreover, from the definition of $\bar{\beta}_k^i$, it follows that
$0 \le \bar{\beta}_k^i - \tilde\alpha_{k+1}^i \le (1/\delta-1)\tilde\alpha_{k+1}^i \le \tilde\alpha_{k+1}^i/\delta$, where we have used the fact that $\delta \in (0,1)$. Then,
\begin{equation*}
\begin{split}
     -\nabla_i f(\xi_k^i) & \leq 0, \\
    \nabla_i f(\bar\xi_k^i) & < \gamma \left(\dfrac{1-\delta}{\delta} \right) \tilde\alpha_{k+1}^i \leq \gamma \frac{\tilde\alpha_{k+1}^i}{\delta},
\end{split}
\end{equation*}
and, recalling that $t_k^i, \bar t_k^i \in (0,1)$, we also have
\begin{align*}
\|y_k^i-\xi_k^i\| & = \beta_k^i+t_k^i(\tilde\alpha_{k+1}^i-\beta_k^i)\leq \tilde\alpha_{k+1}^i \le \frac{\tilde\alpha_{k+1}^i}{\delta},\\
\|y_k^i-\bar\xi_k^i\| & = \tilde\alpha_{k+1}^i+t_k^i(\bar\beta_k^i-\tilde\alpha_{k+1}^i) \leq \bar \beta_k^i \leq \frac{\tilde\alpha_{k+1}^i}{\delta}.
\end{align*}
Hence, reasoning as in the previous case, we obtain
\begin{equation}\label{prop:baoundgr6}
\begin{split}
-\nabla f(x_k)^\top (-e_i)&= \nabla_i f(x_k) \le \left(\dfrac{\gamma+L}{\delta}+L\sqrt{n}\right) \Delta_{k+1},\\
-\nabla f(x_k)^\top e_i&= -\nabla_i f(x_k) \le \left(\dfrac{\gamma+L}{\delta}+L\sqrt{n}\right) \Delta_{k+1}.
\end{split}
\end{equation}
\end{itemize}
Hence, from~\eqref{prop:baoundgr3}, \eqref{prop:baoundgr4}, \eqref{prop:baoundgr3_u}, \eqref{prop:baoundgr4_u}, \eqref{prop:baoundgr3_ul} and~\eqref{prop:baoundgr6}, we conclude that

\[
-\nabla f(x_k)^\top d \le \left(\dfrac{\gamma+L}{\delta}+L\sqrt{n}\right) \Delta_{k+1} \quad \forall d\in G_k.
\]

Now, let us analyze an iteration $k$ such that $x_{k+1}= x_k$ (i.e., an unsuccessful iteration). In such a case, only cases (1a), (2a) and (3a) can occur, although we have to consider that $y_k^i = x_k$ and $\tilde \alpha_{k+1}^i = \theta \nu_k^i$.
Hence, replacing $L$ with $L_i$, we have that the following relations hold:
\begin{align*}
-\nabla f(x_k)^\top e_i & \le \frac{\gamma+\Lmax}{\theta} \Delta_{k+1}\quad\text{if}\ e_i\in G_k, \\
-\nabla f(x_k)^\top (-e_i) & \le \frac{\gamma+\Lmax}{\theta} \Delta_{k+1}\quad\text{if}\ -e_i\in G_k.
\end{align*}
As above, we finally conclude
\[
-\nabla f(x_k)^\top d \le \frac{\gamma+\Lmax}{\theta} \Delta_{k+1} \quad \forall d\in G_k.
\]
\end{proof}

\begin{remark}\label{rem:ls}
The result expressed in Theorem~\ref{bound_grad}
strongly relies on the extrapolation phase of the line search procedure (i.e., lines 12--15 of Algorithm~\ref{alg:ls}).
In particular, since we quit the expansion with a failure in the objective decrease when we do not hit the border of the feasible set, then we are able to upper bound $\chi(x_k)$ for all iterations, including the successful ones.
This represents a relevant difference over direct-search~\cite{gratton2019direct} methods, where $\chi(x_k)$ is usually upper bounded only for the unsuccessful iterations.
Moreover, this property will allow us to give a complexity bound on the total number of iterations where $\chi(x_k)$ is above a prespecified threshold (see Theorem~\ref{th:complexity_iter} below).
\end{remark}

Combining Proposition~\ref{prop:gd} and Theorem~\ref{bound_grad}, it is now straightforward to relate $\chi(x_k)$ with $\Delta_{k+1}$, as stated in the following result.

\begin{theorem}\label{th:chi}
Let $\{x_k\}$ be the sequence of points produced by Algorithm~\ref{alg:dfl}. Then,
\[
\chi(x_k) \le
\begin{cases}
\sqrt{n}\left(\dfrac{\gamma+L}{\delta} + L\sqrt{n} +\dfrac{M_g}{\theta}\right)\Delta_{k+1} \quad & \text{if } x_{k+1} \ne x_k, \\[1.2em]
\sqrt{n}\left(\dfrac{\gamma+\Lmax+M_g}{\theta}\right)\Delta_{k+1} \quad & \text{if } x_{k+1} = x_k.
\end{cases}
\]
\end{theorem}

\begin{proof}
Using Proposition~\ref{prop:chi_ub} with $\epsilon = \Delta_k$ and
Lemma~\ref{lemmalegame}, we obtain
\[
\chi(x_k) \leq \|(-\nabla f(x_k))_{T_k}\| + \frac{\Delta_{k+1}}{\theta}\sqrt{n}M_g,
\]
where we have used the fact that
$\|(-\nabla f(x_k))_{N(x_k,\epsilon)}\| \le \|\nabla f(x_k)\| \le M_g$, where the last inequality follows from Assumption~\ref{assump2}.
So, using Proposition~\ref{prop:gd}, it follows that there exists a direction $d\in G_k$ such that
\[
\chi(x_k) \leq \sqrt{n}\left(-\nabla f(x_k)^\top d + \frac{M_g}{\theta}\Delta_{k+1}\right).
\]
The desired result hence follows from Theorem~\ref{bound_grad}.
\end{proof}

In order to get convergence to stationary points and provide worst-case complexity bounds for the proposed algorithm,
for each iteration $k$ let us define
\begin{equation}\label{Phi_def}
\Phi_k = f(x_k) + {\eta} \Delta_k^2,
\end{equation}
{where $\eta$ satisfies
\begin{equation}\label{eq:eta_ineq}
    0 < \eta < \gamma(\delta(1-\delta))^2.
\end{equation}
}
Recalling Assumption~\ref{assump1}, note that
\begin{equation}\label{eq:lowerbound_Phi}
\Phi_k \geq f_{\min} \quad \forall k \ge 0.
\end{equation}
Now, in the next theorem, we bound the difference $\Phi_{k} -\Phi_{k-1}$ for each iteration $k$.

\begin{theorem}\label{teo:Phi}
Let $\{x_k\}$ be the sequence produced by Algorithm~\ref{alg:dfl}. Then,
\begin{equation}\label{eq:bound_Phikkm1}
\Phi_{k} -\Phi_{k-1} \leq -c_1 \Delta_k^2 \quad \forall k \ge 1,
\end{equation}
 with
\begin{equation}\label{ctilde}
 c_1 = \min \left\{ \gamma c^2, \gamma(\delta(1-\delta))^2-\eta,\eta\left(  \frac{1 - \theta^2}{\theta^2}\right) \right\} > 0,
\end{equation}
where $c$ and $\theta$ are defined in Algorithm~\ref{alg:dfl}.
\end{theorem}

\begin{proof}
For each iteration $k\ge 1$, consider the iteration $k-1$. The following cases can occur.
\begin{itemize}
\item $x_k \neq x_{k-1}$ (i.e., $k-1$ is a successful iteration) and
$\Delta_k = \Delta_{k-1}$.

Hence, there exists a coordinate ${\bar \imath}$ such that
we have moved from $x_{k-1}$ along $\pm e_{\bar \imath}$ and we have
\[
  \alpha_{k-1}^{\bar\imath} \geq \nu_{k-1}^{\bar\imath} = \max\left\{\tilde\alpha_{k-1}^{\bar\imath},c\Delta_{k-1}\right\} \geq c\Delta_{k-1} = c\Delta_k.
\]

Then, from the line search procedure, we can write
\[
f(x_k) \leq f(x_{k-1}) -\gamma c^2 \Delta_k^2.
\]
Hence, we have
{
\begin{equation}\label{bound_Phik_case1}
\Phi_k - \Phi_{k-1} = f(x_k) - f(x_{k-1}) + \eta \left(\Delta_k^2 - \Delta_{k-1}^2 \right) \leq - \gamma c^2 \Delta_k^2.
\end{equation}
}

\item $x_k \neq x_{k-1}$ (i.e., $k-1$ is a successful iteration) and $\Delta_k > \Delta_{k-1}$.
Hence, there exists a coordinate ${\bar \imath}$ such that
\[
\alpha_{k-1}^{\bar\imath} = \Delta_k.
\]
{{
Let $h_{k-1}^i$ be the number of times the stepsize related to the $i$th coordinate is expanded on iteration $k-1$, that is, if $\alpha_k^i>0$
\[
\alpha_{k-1}^i = \min\{\alphamax,\delta^{-h_{k-1}^i}\nu_{k-1}^i\},
\]
where $\alphamax$ is the distance to the bound of the $i$th coordinate over $d_{k-1}^i$. Now, considering $\bar{\imath}$, we have
\[
\alpha_{k-1}^{\bar\imath}>\Delta_{k-1}\geq\nu_{k-1}^{\bar\imath}.
\]
In particular, the first inequality implies that $h_{k-1}^{\bar\imath}\geq1$ (otherwise we would have $\alpha_{k-1}^{\bar\imath}\in\{0,\nu_{k-1}^{\bar\imath}\}$).
If $h_{k-1}^{\bar\imath}=1$, then we have
\[
\Delta_k = \alpha_{k-1}^{\bar\imath}=\min\{\alphamax,\delta^{-1}\nu_{k-1}^{\bar\imath}\}\leq \delta^{-1}\nu_{k-1}^{\bar\imath},
\]
that is,
\[
\nu_{k-1}^{\bar\imath} \geq \delta \Delta_k.
\]
Thus, we get
\begin{equation}\label{eq:case_h1}
    f(x_k)\leq f(x_{k-1})-\gamma(\nu_{k-1}^{\bar\imath})^2\leq f(x_{k-1})-\gamma\delta^2\Delta_k^2.
\end{equation}
Let us now consider $h_{k-1}^{\bar\imath}\geq 2$. We have
\[
\Delta_k=\alpha_{k-1}^{\bar \imath}=\min\{\alphamax,\delta^{-h_{k-1}^{\bar\imath}}\nu_{k-1}^{\bar \imath}\}\leq \delta^{-h_{k-1}^{\bar\imath}}\nu_{k-1}^{\bar \imath},
\]
that is,
\begin{equation}\label{eq:delta_for_hplus2}
    \delta^{-(h_{k-1}^{\bar \imath}-1)}\nu_{k-1}^{\bar \imath}\geq\delta\Delta_k.
\end{equation}
Therefore, we can write
\begin{equation}\label{eq:case_h2plus}
    \begin{split}
        \displaystyle f(x_k)&\leq f(x_{k-1})-\gamma\left(\delta^{-(h_{k-1}^{\bar \imath}-1)}\nu_{k-1}^{\bar \imath}-\delta^{-(h_{k-1}^{\bar \imath}-2)}\nu_{k-1}^{\bar \imath}\right)^2\\
        &= f(x_{k-1})-\gamma\left(\delta^{-(h_{k-1}^{\bar \imath}-1)}\nu_{k-1}^{\bar \imath}(1-\delta)\right)^2\\
        &\leq f(x_{k-1})-\gamma\left(\delta(1-\delta)\right)^2\Delta_k^2,
    \end{split}
\end{equation}
where the last inequality follows from~\eqref{eq:delta_for_hplus2}.
Finally, using~\eqref{eq:case_h1} and~\eqref{eq:case_h2plus}, and considering that $\min\{(1-\delta)^2,(\delta(1-\delta))^2\}= (\delta(1-\delta))^2$, we can write
\[
f(x_k) \leq f(x_{k-1}) -\gamma(\delta(1-\delta))^2\Delta_k^2.
\]
Hence, we have
\begin{equation}\label{bound_Phik_case2}
\begin{split}
\Phi_k - \Phi_{k-1} & = f(x_k) - f(x_{k-1}) + \eta \left(\Delta_k^2 - \Delta_{k-1}^2 \right)  \\
  &\leq -\gamma(\delta(1-\delta))^2\Delta_k^2 + \eta\Delta_k^2 \\
  & = - \left(\gamma(\delta(1-\delta))^2-\eta\right)\Delta_k^2.
\end{split}
\end{equation}
}}

\item $x_k = x_{k-1}$ (i.e., $k-1$ is an unsuccessful iteration). By Lemma \ref{lemmalegame}, we have
\[
\Delta_k = \theta \Delta_{k-1}.
\]
Then,
\begin{equation}\label{bound_Phik_case3}
\begin{split}
\Phi_k - \Phi_{k-1} &= f(x_k)  - f(x_{k-1}) + \eta\left( \Delta_k^2 - \Delta_{k-1}^2\right)  \\
  &= \eta\left( \Delta_k^2 -\frac{1}{\theta^2}\Delta_k^2\right) \\
  & = -\eta\left(  \frac{1 - \theta^2}{\theta^2}\right)\Delta_k^2.
\end{split}
\end{equation}
\end{itemize}
Finally, we get \eqref{eq:bound_Phikkm1} by combining~\eqref{bound_Phik_case1}, \eqref{bound_Phik_case2} and~\eqref{bound_Phik_case3}, where $c_1 > 0$ since $c\in (0,1]$.
\end{proof}

Using the above result, we can easily show the convergence to zero of the sequences of tentative and actual stepsizes produced by the algorithm, i.e., $\{\tilde \alpha_k^i\}$, $\{\alpha_k^i\}$, $i=1,\dots,n$, respectively,
together with the convergence to zero of the sequence of maximum tentative stepsizes $\{\Delta_k\}$.

\begin{proposition}\label{prop:conv1}
Let $\{\tilde\alpha_k^i\}$, $\{\alpha_k^i\}$, $i=1,\dots,n$, and $\{\Delta_k\}$ be the sequences produced by Algorithm~\ref{alg:dfl}. We have that
\begin{align}
\lim_{k \to \infty} \tilde \alpha^i_k = 0, & \quad i=1,\ldots,n; \label{alfatildetozero}\\
\lim_{k \to \infty} \alpha_k^i = 0, & \quad i=1,\ldots,n;\label{alfatozero} \\
\lim_{k \to \infty} \Delta_k = 0. & \label{Deltatozero}
\end{align}
\end{proposition}

\begin{proof}
From~\eqref{eq:lowerbound_Phi} and~\eqref{eq:bound_Phikkm1}, we get~\eqref{Deltatozero}.
Then, using the definition of $\Delta_k$ given in Algorithm~\ref{alg:dfl}, also~\eqref{alfatildetozero} holds.
Since, for all $k \ge 1$ and for all $i = 1,\ldots,n$, from the instructions of the algorithm
either $\alpha^i_{k-1} = 0$ or $\alpha^i_{k-1} = \tilde \alpha^i_k$, then \eqref{alfatozero} follows from \eqref{alfatildetozero}.
\end{proof}

Using Theorem~\ref{th:chi} and Proposition~\ref{prop:conv1}, it is now possible to prove the convergence of the algorithm to stationary points.

\begin{theorem}\label{th:stationary}
Let $\{x_k\}$ be the sequence of points produced by Algorithm~\ref{alg:dfl}. Then,
\begin{itemize}
\item $\displaystyle\lim_{k\to\infty}\chi(x_k) = 0$, i.e., every limit point of $\{x_k\}$ is stationary,
\item $\displaystyle{\lim_{k \to \infty} \|x_{k+1}-x_k\| = 0}$.
\end{itemize}
\end{theorem}

\begin{proof}
From Theorem~\ref{th:chi} and \eqref{Deltatozero} in Proposition~\ref{prop:conv1}, taking the limit for $k\to\infty$ it follows that $\chi(x_k) \to 0$, that is, every limit point of $\{x_k\}$ is stationary.
Finally, since $x_{k+1} - x_k = \sum_{i=1}^n \alpha_k^id_k^i$, using~\eqref{alfatozero} in Proposition~\ref{prop:conv1} we also have that $\|x_{k+1}-x_k\| \to 0$.
\end{proof}

\subsection{Worst-case complexity}\label{susec:worst_case}
This section is devoted to analyze the worst-case complexity of Algorithm~\ref{alg:dfl}. In particular,
\begin{itemize}
    \item[(i)] we give an upper bound of $\bigO{(n^2 \epsilon^{-2})}$ on the total number of iterations where $\chi (x_k)$ is not below a prespecified threshold $\epsilon$;
    \item[(ii)] we give an upper bound of $\bigO{(n \epsilon^{-2})}$ on the number of iterations required to generate the first point $x_k$ where $\chi (x_k)$ is below a prespecified tolerance $\epsilon$;
    \item[(iii)] we give an upper bound of  $\bigO{(n^2 \epsilon^{-2})}$ on the number of function evaluations required to generate the first point $x_k$ where $\chi (x_k)$ is below a prespecified tolerance $\epsilon$.
\end{itemize}

We start by providing an upper bound of  $\bigO{(n^2 \epsilon^{-2})}$ on the total number of iterations where $\chi (x_k) \geq \epsilon$, with a given $ \epsilon > 0$.

\begin{theorem}\label{th:complexity_iter}
Let $\{x_k\}$ be the sequence of points produced by the algorithm. Given any $\epsilon > 0$, let
\[
K_{\epsilon}=\{k \colon \chi(x_k) \ge \epsilon \}.
\]
Then, $|K_\epsilon| \leq \bigO{(n^2\epsilon^{-2})}$.
In particular,
\[
|K_\epsilon| \leq \left\lfloor\frac{c_2^2 (\Phi_0 - f_{\min})}{c_1}\epsilon^{-2}\right\rfloor,
 \]
where $c_1$ is defined as in Theorem~\ref{teo:Phi} and
\[
 c_2 = \sqrt{n} \max\left\{\frac{\gamma+L}{\delta} +L\sqrt{n} + \frac{M_g}{\theta},\dfrac{\gamma+\Lmax+M_g}{\theta}\right\}.
\]
\end{theorem}

\begin{proof}
From Theorem \ref{teo:Phi}, it follows that the sequence $\{\Phi_k\}$ is monotonically non-increasing. Furthermore, for all $k \ge 1$, we can write
\begin{equation}\label{complessitaset1}	
\Phi_{k} - \Phi_{0} \leq  - c_1 \sum_{j=1}^{k}\Delta_j^2 = - c_1 \sum_{j=0}^{k-1}\Delta_{j+1}^2.
\end{equation}
Since the sequence $\{\Phi_k\}$ is bounded from below, there exists $\Phi^*$ such that
\[
\lim_{k\to\infty}\Phi_k=\Phi^*\ge f_{\min},
\]
with $f_{\min}$ defined as in~\eqref{eq:lowerbound_Phi}.
Taking the limit for $k\to\infty$ in~\eqref{complessitaset1} we obtain
	\[
	\Phi_0 - f_{\min} \ge \Phi_0 -\Phi^*\geq  c_1 \sum_{k=0}^{\infty}\Delta_{k+1}^2\geq c_1 \sum_{k\in K_\epsilon}\Delta_{k+1}^2.
	\]
Therefore, using the definition of $K_\epsilon$ and Theorem~\ref{th:chi},
we get
\[
\Phi_0 - f_{\min} \geq c_1 \sum_{{k\in K_\epsilon}}\Delta_{k+1}^2 \geq |K_\epsilon| c_1\frac{\epsilon^2}{c_2^2}.
\]
Thus, the desired result is obtained.
\end{proof}

As appears from the proof of Theorem~\ref{th:complexity_iter}, the above result relies on Theorem~\ref{th:chi} which, in turn, uses Theorem~\ref{grad_app}.
The latter, as pointed out in Remark~\ref{rem:ls},
strongly relies on the extrapolation phase of the line search procedure (i.e., lines 12--15 of Algorithm~\ref{alg:ls}),
which allows us to bound $\chi(x_k)$ at both successful and unsuccessful iterations.

In the following theorem, we give an upper bound of $\bigO{(n\epsilon^{-2})}$ on the maximum number of iterations required to
produce a point $x_k$ such that $\chi(x_k)$ is below a given threshold $\epsilon > 0$.
This bound aligns with established findings for direct-search~\cite{gratton2019direct} and  model-based~\cite{hough2022model} methods.

\begin{theorem}\label{th:compl}
Let $\{x_k\}$ be the sequence of points produced by Algorithm~\ref{alg:dfl}.
Given any $\epsilon > 0$, let $j_{\epsilon} \ge 1$ be the first iteration such that $\chi(x_{j_{\epsilon}})<\epsilon$, that is, $\chi(x_k) \geq \epsilon$ for all $k \in \{0,\dots,j_{\epsilon}-1\}$.

Then, $j_{\epsilon} \leq \bigO{(n\epsilon^{-2})}$. In particular,
\[
j_{\epsilon} \leq \biggl\lfloor\frac{n c_3^2 (\Phi_0 - f_{\min})}{c_1}\epsilon^{-2}\biggr\rfloor,
\]
where $c_1$ is given by~\eqref{ctilde} and
\begin{equation}\label{c1def_LAM}
c_3 = \frac{\gamma+\Lmax + M_g}{\theta}.
\end{equation}	
\end{theorem}
\begin{proof}
Let $\Phi_k$ the function defined in~\eqref{Phi_def}. We can write
\[
\Phi_{j_{\epsilon}} - \Phi_{0} = \sum_{k=0}^{j_{\epsilon}-1} (\Phi_{k+1}-\Phi_k)
\]
and, using Theorem~\ref{teo:Phi}, we have that
\[
\Phi_{j_{\epsilon}} - \Phi_{0} \leq  - c_1 \sum_{k=0}^{j_{\epsilon}-1}\Delta_{k+1}^2.
\]
Recalling~\eqref{eq:lowerbound_Phi}
and the fact that $\Phi_k \ge f(x_k)$ for $k \ge 0$, we get
\begin{equation}\label{eq:bound_diff_Phi}
f_{\min} -\Phi_0 \leq \Phi_{j_{\epsilon}} - \Phi_{0}\leq - c_1 \sum_{k=0}^{j_{\epsilon}-1}\Delta_{k+1}^2.
\end{equation}
Now, we can partition the set of iteration indices $\{0,\dots,j_{\epsilon}-1\}$ into $\mathcal S_{j_{\epsilon}}$ and $\mathcal U_{j_{\epsilon}}$ such that
\[
k \in \mathcal S_{j_{\epsilon}} \, \Leftrightarrow \,  x_k \not= x_{k-1}, \quad k\in \mathcal U_{j_{\epsilon}} \, \Leftrightarrow \,  x_k = x_{k-1}, \quad \mathcal S_{j_{\epsilon}} \cup \mathcal U_{j_{\epsilon}} = \{0,\dots,j_{\epsilon}-1\},
\]
that is, $\mathcal S_{j_{\epsilon}}$ and $\mathcal U_{j_{\epsilon}}$ contain the successful and unsuccessful iterations up to $j_{\epsilon}-1$, respectively.
So, from~\eqref{eq:bound_diff_Phi}, we can write
\begin{equation}\label{eq:bound_diff_Phi2}
\Phi_0-f_{\min} \geq c_1 \sum_{k\in \mathcal S_{j_{\epsilon}}}\Delta_{k+1}^2 + c_1 \sum_{k\in \mathcal U_{j_{\epsilon}}}\Delta_{k+1}^2.
\end{equation}
For all $k\in \mathcal S_{j_{\epsilon}}$, let us define the index $m(k)$ as follows:
\begin{itemize}
\item if $\mathcal U_{j_{\epsilon}}\cap \{0,\ldots,k-1\}\not= \emptyset$, then $m(k)$ in the largest index of $\mathcal U_{j_{\epsilon}}\cap \{0,\ldots,k-1\}$;
\item otherwise, $m(k) = -1$.
\end{itemize}
Note that, by definition, $m(k)$ is the last unsuccessful iteration before iteration $k$, i.e., all the iterations from $m(k)+1$ to $k$ are successful iterations.
Lemma~\ref{lemmalegame} guarantees that $\Delta_{k+1} \ge \Delta_{m(k)+1}$ for all $k\in \mathcal S_{j_{\epsilon}}$.
Using~\eqref{eq:bound_diff_Phi2}, we obtain
\[
\Phi_0-f_{\min} \geq c_1 \sum_{k\in \mathcal S_{j_{\epsilon}}}\Delta_{m(k)+1}^2 + c_1 \sum_{k\in \mathcal U_{j_{\epsilon}}}\Delta_{k+1}^2.
\]
From Theorem~\ref{th:chi}, we have that
$\Delta_{m(k)+1} \ge \chi (x_{m(k)})/ (\sqrt n c_3)$ for all $k\in \mathcal S_{j_{\epsilon}}$
and
$\Delta_{k+1} \ge \chi (x_k)/ (\sqrt n c_3)$ for all $k\in \mathcal U_{j_{\epsilon}}$.
Since $\chi (x_k) \geq \epsilon$ for all $k \in \{0,\dots,j_{\epsilon}-1\}$,
with $\mathcal S_{j_{\epsilon}} \cup \mathcal U_{j_{\epsilon}} = \{0,\dots,j_{\epsilon}-1\}$, we get
\[
\Phi_0 - f_{\min} \geq j_{\epsilon} \frac{c_1}{nc_3^2} \epsilon^2.
\]
Thus, the desired result is obtained.
\end{proof}

The last complexity result we give is about the maximum number of function evaluations required to
produce a point $x_k$ such that $\chi(x_k)$ is less than or equal to a given threshold $\epsilon > 0$.
Using arguments from the related literature~\cite{brilli2024worst,vicente2013worst}, we obtain an upper bound of $\bigO{(n^2 \epsilon^{-2})}$, which still aligns with established findings for direct-search~\cite{gratton2019direct} and  model-based~\cite{hough2022model} methods

\begin{theorem}
Let $\{x_k\}$ be the sequence of points produced by Algorithm~\ref{alg:dfl}.
Given any $\epsilon > 0$, let $j_{\epsilon} \ge 1$ be the first iteration such that $\chi(x_{j_{\epsilon}}) < \epsilon$, that is, $\chi(x_k) \geq \epsilon$ for all $k \in \{0,\dots,j_{\epsilon}-1\}$.

Denoting by $Nf_{j_{\epsilon}}$ the number of function evaluations required by Algorithm~\ref{alg:dfl} up to iteration $j_{\epsilon}$, then
$Nf_{j_{\epsilon}} \leq \bigO{(n^2\epsilon^{-2})}$. In particular,
\[
Nf_{j_{\epsilon}} \leq 2n \left\lfloor\frac{ nc_3^2 (\Phi_0 - f_{\min})}{c_1}\epsilon^{-2}\right\rfloor + \left\lfloor\frac{nc_3^2(f_0 - f_{\min})}{\gamma c^2}\max\Bigg\{1,\Biggl(\frac{\delta}{1-\delta}\Biggr)^2\Bigg\}\epsilon^{-2}\right\rfloor,
\]
where $c_1$ and $c_3$ are given in~\eqref{ctilde} and \eqref{c1def_LAM}, respectively.
\end{theorem}

\begin{proof}
First, let us partition the set of iteration indices $\{0,\dots,j_{\epsilon}-1\}$ into $\mathcal S_{j_{\epsilon}}$ and $\mathcal U_{j_{\epsilon}}$ such that
\[
k \in \mathcal S_{j_{\epsilon}} \, \Leftrightarrow \,  x_k \not= x_{k-1}, \quad k\in \mathcal U_{j_{\epsilon}} \, \Leftrightarrow \,  x_k = x_{k-1}, \quad \mathcal S_{j_{\epsilon}} \cup \mathcal U_{j_{\epsilon}} = \{0,\dots,j_{\epsilon}-1\},
\]
that is, $\mathcal S_{j_{\epsilon}}$ and $\mathcal U_{j_{\epsilon}}$ contain the successful and unsuccessful iterations up to $j_{\epsilon}-1$, respectively.

When the algorithm evaluates a new point, the latter can either succeed to decrease the objective function or fail to do so. Let us then define $Nf_{j_{\epsilon}}^{\mathcal S}$ as the total number of function evaluations related to points which succeed to decrease the objective function up to iteration $j_{\epsilon}$. Note that, at each iteration, the maximum number of function evaluations related to points which fail to decrease the objective function is $2n$ (and it can be equal to $2n$ only when $T_k = \R^n$).
So, we can write
\begin{equation}\label{NF}
Nf_{j_{\epsilon}} \leq 2n j_{\epsilon} + Nf_{j_{\epsilon}}^{\mathcal S} \le 2n \biggl\lfloor\frac{n c_3^2 (\Phi_0 - f_{\min})}{c_1}\epsilon^{-2}\biggr\rfloor + Nf_{j_{\epsilon}}^{\mathcal S},
\end{equation}
where the last inequality follows from Theorem~\ref{th:compl}.
Now, let us consider any iteration $k < j_{\epsilon}$ and any index $i \in \{1,\ldots,n\}$
such that the line search succeeds to produce a decrease in the objective function.
For each $\alpha$ used in the extrapolation phase of the line search, we have that either
\begin{equation}\label{eq:bound_diff_f_first}
f(y_k^i) - f(y_k^i + \alpha d_k^i) \geq \gamma \alpha^2\geq \gamma c^2\Delta_k^2,
\end{equation}
or
\begin{equation}\label{eq:bound_diff_f}
f(y_k^i + \alpha d_k^i) - f(y_k^i + (\alpha/\delta) d_k^i) \geq \gamma \Biggl(\frac{1-\delta}{\delta}\Biggr)^2\alpha^2\geq \gamma \Biggl(\frac{1-\delta}{\delta}\Biggr)^2 c^2\Delta_k^2.
\end{equation}
Let us define the index $m(k)$ as follows:
\begin{itemize}
\item if $\mathcal U_{j_{\epsilon}}\cap \{0,\ldots,k-1\}\not= \emptyset$, then $m(k)$ is the largest index of $\mathcal U_{j_{\epsilon}}\cap \{0,\ldots,k-1\}$;
\item otherwise, $m(k) = 0$.
\end{itemize}
Note that, by definition, $m(k)$ is the last unsuccessful iteration before iteration $k$, i.e., all the iterations from $m(k)+1$ to $k$ are successful iterations.
Lemma~\ref{lemmalegame} guarantees that $\Delta_k \ge \Delta_{m(k)+1}$ for all $k\in \mathcal S_{j_{\epsilon}}$.
Hence, from~\eqref{eq:bound_diff_f_first} and~\eqref{eq:bound_diff_f}, it follows that
\begin{gather*}
f(y_k^i) - f(y_k^i + \alpha d_k^i) \geq \gamma \alpha^2\geq \gamma c^2\Delta_{m(k)+1}^2 \geq \gamma \min\Bigg\{1,\Biggl(\frac{1-\delta}{\delta}\Biggr)^2\Bigg\} c^2\Delta_{m(k)+1}^2 \\
f(y_k^i + \alpha d_k^i) - f(y_k^i + (\alpha/\delta) d_k^i) \geq \gamma \Biggl(\frac{1-\delta}{\delta}\Biggr)^2\alpha^2\geq \gamma  {\min\Bigg\{1,\Biggl(\frac{1-\delta}{\delta}\Biggr)^2\Bigg\}} c^2\Delta_{m(k)+1}^2.
\end{gather*}
From Theorem~\ref{th:chi}, we have that $\Delta_{m(k)+1} \ge \chi (x_{m(k)})/ (\sqrt n c_3)$.
Since $\chi (x_k) \geq \epsilon$ for all $k \in \{0,\dots,j_{\epsilon}-1\}$,
we can write
\[
f(y_k^i + \alpha d_k^i) - f(y_k^i + (\alpha/\delta) d_k^i) \geq
\gamma \min\Bigg\{1,\Biggl(\frac{1-\delta}{\delta}\Biggr)^2\Bigg\} c^2\frac{\epsilon^2}{nc_3^2}.
\]
Then, recalling Assumption~\ref{assump1} and summing up the above relation over all function evaluations producing an objective decrease, we obtain
\[
f_0 - f_{\min} \geq Nf_{j_{\epsilon}}^{\mathcal S}\gamma {\min\Bigg\{1,\Biggl(\frac{1-\delta}{\delta}\Biggr)^2\Bigg\}} c^2\frac{\epsilon^2}{nc_3^2},
\]
that is,
\[
Nf_{j_{\epsilon}}^{\mathcal S} \leq \left\lfloor\frac{nc_3^2(f_0 - f_{\min})}{\gamma c^2}{\max\Bigg\{1,\Biggl(\frac{\delta}{1-\delta}\Biggr)^2\Bigg\}}\epsilon^{-2}\right\rfloor.
\]
The desired results hence follows from~\eqref{NF}.
\end{proof}

\section{Finite active-set identification}\label{sec4}
In this section, we show that Algorithm~\ref{alg:dfl} identifies the components of the final solution lying on the lower or the upper bounds (the so called \textit{active set}) in a finite number of iterations.

First, let us give an equivalent definition of stationarity for problem~\eqref{prob}, which will be useful in our analysis.

\begin{definition}\label{def:stat_grad}
A point $x^* \in [l,u]$ is said to be  a stationary point of problem~\eqref{prob} (i.e., $\chi(x^*)=0$) if, for all $i \in \{1,\ldots,n\}$, we have that
\[
\nabla_i f(x^*)
\begin{cases}
\geq 0 \quad \text{if } x^*_i = l_i, \\
=0 \quad \text{if } l_i < x^*_i < u_i, \\
\leq 0 \quad \text{if } x^*_i = u_i.
\end{cases}
\]
\end{definition}

Now, let us recall the definition of \textit{strict complementarity} and \textit{non-degenerate solutions}.

\begin{definition}\label{def:sc}
Given a stationary point $x^*$ of problem~\eqref{prob}, we say that a component $x^*_i$ satisfies the strict complementarity condition if $x^*_i \in \{l_i,u_i\}$ and $\nabla_i f(x^*) \ne 0$.
If the strict complementarity condition is satisfied by all components $x^*_i$, we say that $x^*$ is non-degenerate.
\end{definition}

In particular, we define $\Astar(x^*)$ as the \textit{active set} for a stationary point $x^*$, that is,
the index set for the active components of $x^*$.
We also define $\Astar^+(x^*)$ as the index set for those components satisfying the strict complementarity condition. Namely,
\[
\Astar(x^*) = \{i \colon x^*_i = l_i\} \cup  \{i \colon x^*_i = u_i\} \qquad \text{and} \qquad
\Astar^+(x^*) = \Astar(x^*) \cap \{i \colon \nabla_i f(x^*) \ne 0\}.
\]

Furthermore, for any stationary point $x^*$ such that $\Astar^+(x^*) \ne \emptyset$, let us define
\begin{equation}\label{sc_measure}
\zeta(x^*) = \min_{i \in \Astar^+(x^*)} \abs{\nabla_i f(x^*)}.
\end{equation}
We see that $\zeta(x^*)$ is a measure of the minimum amount of strict complementarity among the variables in $\Astar^+(x^*)$.
This quantity will be used to define a neighborhood of $x^*$ where the active components are correctly identified, following a similar approach as in~\cite{cristofari2022active,nutini2019active}.

Before diving into the main theorem of this section, we need some preliminary results following from the Lipschitz continuity of $\nabla f$. Recalling Definition~\ref{coord_lip}, using standard arguments (see, e.g.,~\cite{nesterov2013introductory}) one can prove that, for all $x \in [l,u]$, we have
\[
\abs{f(x+se_i) - f(x) - s \nabla_i f(x)} \le \frac{L_i}2 s^2 \quad \forall s \in \R\colon x+se_i \in [l,u], \quad i=1,\ldots,n.
\]
Hence, for all $x \in [l,u]$, we have
\begin{align}
f(x+se_i) & \le f(x) + s \nabla_i f(x) + \frac{L_i}2 s^2 \quad \forall s \in \R\colon x+se_i \in [l,u], \quad i=1,\ldots,n, \label{lips_descent} \\
f(x+se_i) & \ge f(x) + s \nabla_i f(x) - \frac{L_i}2 s^2 \quad \forall s \in \R\colon x+se_i \in [l,u], \quad i=1,\ldots,n. \label{lips_ascent}
\end{align}

The two following results provide bounds for the objective function when exploring any coordinate direction.

\begin{proposition}\label{prop:alpha_lips_descent}
Given $x \in [l,u]$, $\gamma \ge 0$ and $i \in \{1,\ldots,n\}$,
then
\[
f(x - s \sign(\nabla_i f(x)) e_i) \le f(x) - \gamma s^2
\]
for all $0 \le s \le 2\dfrac{\abs{\nabla_i f(x)}}{L_i+2\gamma}$ such that $x - s \sign(\nabla_i f(x)) e_i \in [l,u]$.
\end{proposition}

\begin{proof}
From~\eqref{lips_descent}, we can write
\[
f(x + s e_i) \le f(x) + s \biggl(\nabla_i f(x) + \frac{L_i}2 s\biggr) \quad \forall s \in \R \colon x+se_i \in [l,u].
\]
The right-hand side of the above inequality is less than or equal to $f(x) - \gamma s^2$ if
\[
s\biggl(\nabla_i f(x) + \frac{L_i}2 s\biggr) \le -\gamma s^2.
\]
If $\nabla_i f(x) \ne 0$, solving with respect to $s$ we obtain
\begin{align*}
-\frac{2\nabla_i f(x)}{L_i+2\gamma} \le s \le 0 \quad & \text{if } \nabla_i f(x) > 0, \\
0 \le s \le -\frac{2\nabla_i f(x)}{L_i+2\gamma} \quad & \text{if } \nabla_i f(x) < 0,
\end{align*}
leading to the desired result.
\end{proof}

\begin{proposition}\label{prop:alpha_lips_ascent}
Given $x \in [l,u]$ and $i \in \{1,\ldots,n\}$,
then
\[
f(x + s \sign(\nabla_i f(x)) e_i) \ge f(x)
\]
for all $0 \le  s \le \dfrac{2\abs{\nabla_i f(x)}}{L_i}$ such that $x + s \sign(\nabla_i f(x)) e_i \in [l,u]$.
\end{proposition}
\begin{proof}
From~\eqref{lips_ascent}, we can write
\[
f(x+se_i) \ge f(x) + s \biggl(\nabla_i f(x) - \frac{L_i}2 s\biggr) \quad \forall s \in \R\colon x+se_i \in [l,u].
\]
The right-hand side of the above inequality is greater than or equal to $f(x)$ if
\[
s\biggl(\nabla_i f(x) - \frac{L_i}2 s\biggr) \ge 0.
\]
If $\nabla_i f(x) \ne 0$, solving with respect to $s$ we obtain
\begin{align*}
0 \le s \le \frac{2\nabla_i f(x)}{L_i} \quad & \text{if } \nabla_i f(x) > 0, \\
\frac{2\nabla_i f(x)}{L_i} \le s \le 0 \quad & \text{if } \nabla_i f(x) < 0,
\end{align*}
leading to the desired result.
\end{proof}

The next proposition shows that, when $\nu^i_k$ is sufficiently small at a given iteration, Algorithm \ref{alg:dfl} cannot move along an ascent direction.

\begin{proposition}\label{prop:d2}
Consider an iteration $k$ of Algorithm~\ref{alg:dfl}.
If $\nu^i_k \le 2\abs{\nabla_i f(y_k^i)}/L_i$ for an index $i \in \{1,\ldots,n\}$, then
\[
\alpha_k^i > 0 \; \Rightarrow \;
d_k^i = -\sign(\nabla_i f(y_k^i)) e_i.
\]
\end{proposition}
\begin{proof}
Using Proposition~\ref{prop:alpha_lips_ascent}, for all $\alpha \le \nu^i_k$ and $\gamma > 0$ we have
\[
f(y_k^i + \alpha \sign(\nabla_i f(y_k^i)) e_i) \ge f(y_k^i) > f(y_k^i) - \gamma \alpha^2.
\]
Thus, the line search in Algorithm~\ref{alg:ls} fails when using the direction $\sign(\nabla_i f(y_k^i)) e_i$ with any stepsize $0 < \alpha \le \nu^i_k$.
So, if the line search returns $\alpha_k^i > 0$, necessarily $d_k^i = -\sign(\nabla_i f(y_k^i)) e_i$.
\end{proof}

Now, we are ready to state the main result of this section, establishing finite active-set identification of Algorithm~\ref{alg:dfl}.

\begin{theorem}\label{main_identif_theorem}
Let $\{x_k\}$ be the sequence of points produced by Algorithm~\ref{alg:dfl} and let $x^*$ be a limit point of $\{x_k\}$, i.e., there exists an infinite subsequence $\{x_k\}_K \to x^*$.
Then,
\begin{itemize}
\item[(i)] $\displaystyle{\lim_{k \to \infty, \, k \in K} x_{k+1} = x^*}$;
\item[(ii)] an iteration $\kA \in K$ exists such that, for all $k \ge \kA$, $k \in K$, we have that $(x_{k+1})_i = x^*_i$ for all $i \in \Astar^+(x^*)$.
\end{itemize}
\end{theorem}

\begin{proof}
Since $y_k^1 = x_k$ and $y_k^{i+1} = x_k + \sum_{j=1}^i \alpha_k^j d_k^j$,
from Proposition~\ref{prop:conv1} and the fact that $\norm{d_k^i} = 1$, $i = 1,\ldots,n+1$,
we have
\begin{equation}\label{lim_y}
\lim_{\substack{k \to \infty \\ k \in K}} y_k^i =
\lim_{\substack{k \to \infty \\ k \in K}} x_k =
x^*, \quad i = 1,\ldots,n+1.
\end{equation}
Since $x_{k+1} = y_k^{n+1}$, then point (i) follows.

To show point (ii), assume that $\Astar^+(x^*) \ne \emptyset$.
Let $\kA \in K$ be the first iteration such that the two following relations hold for all $k \ge \kA$, $k \in K$:
\begin{subequations}\label{neighborhood_A}
\begin{gather}
\norm{y_k^i-x^*} \le \min\biggl\{\frac1L,\frac2{2L + \Lmax + 2\gamma}\biggr\} \zeta(x^*), \quad i = 1,\ldots,n, \label{neighborhood_A1} \\
\norm{y_k^i-x^*} + \frac{\Lmax}{2L} \max_{j=1,\ldots,n} \tilde \alpha^j_k \le \frac{\zeta(x^*)}L, \quad i = 1,\ldots,n. \label{neighborhood_A2}
\end{gather}
\end{subequations}
Note that~\eqref{lim_y} and Proposition~\ref{prop:conv1} imply the existence of $\kA \in K$ such that~\eqref{neighborhood_A} holds for all $k \ge \kA$, $k \in K$.

Consider an index $i \in \Astar^+(x^*)$ and an iteration $k \ge \kA$, $k \in K$.
To prove that $(x_{k+1})_i = x^*_i$, we have to show that $(y^{i+1}_k)_i = x^*_i$ since, from the instructions of the algorithm, $(x_{k+1})_i = (y^{i+1}_k)_i$.
Without loss of generality, assume that $x^*_i = l_i$ (the proof for the case $x^*_i = u_i$ is identical, except for minor changes). So, we have to show that
\begin{equation}\label{identification_proof1}
(y^{i+1}_k)_i = l_i.
\end{equation}

Preliminarily, we want to prove that
\begin{gather}
\abs{z_i-l_i} \le \frac{2 \nabla_i f(z)}{L_i + 2\gamma} \quad \forall z \text{ such that } \|z-x^*\| \le \|y_k^i-x^*\|, \label{neigh1}  \\
\nu^i_k \le \frac{2 \nabla_i f(y_k^i)}{L_i}. \label{alpha_tilde_proof}
\end{gather}
According to Definitions~\ref{def:stat_grad} and~\ref{def:sc}, we have that
$\nabla_i f(x^*) > 0$ and, from the definition of $\zeta(x^*)$ given in~\eqref{sc_measure}, it follows that
\begin{equation}\label{zeta_nabla1}
0 < \zeta(x^*) \le \nabla_i f(x^*).
\end{equation}
Consider any $z \in \R^n$ such that
$\|z-x^*\| \le \|y_k^i-x^*\|$.
Using the Lipschitz continuity of $\nabla f$, we have
\begin{equation}\label{lips_proof1}
\nabla_i f(x^*) - \nabla_i f(z) \le \norm{\nabla f(z) - \nabla f(x^*)} \le L \norm{z-x^*}.
\end{equation}
Moreover, from~\eqref{neighborhood_A1} we can write
\[
\|z-x^*\| \le \|y_k^i-x^*\| \le \frac{2\zeta(x^*)}{2L + \Lmax + 2\gamma}.
\]
Multiplying the first and last terms above by $(2L+\Lmax+2\gamma)/(\Lmax+2\gamma)$, we have
\[
\Bigl(\frac{2L}{\Lmax+2\gamma} + 1\Bigr) \norm{z-x^*} \le \frac{2\zeta(x^*)}{\Lmax + 2\gamma},
\]
that is,
\begin{equation}\label{norm_zx}
    \norm{z-x^*} \le \bigl(\zeta(x^*) - L\norm{z-x^*}\bigr) \frac2{\Lmax + 2\gamma}.
\end{equation}
Since, from~\eqref{zeta_nabla1} and~\eqref{lips_proof1}, we have
\begin{equation}\label{zeta_proof}
\zeta(x^*) \le \nabla_i f(z) + L \norm{z-x^*},
\end{equation}
then, using \eqref{norm_zx}, we obtain
\[
\norm{z-x^*} \le \frac{2 \nabla_i f(z)}{\Lmax + 2\gamma}.
\]
Taking into account that $L_i \le \Lmax$ and recalling that $x^*_i = l_i$, it follows that~\eqref{neigh1} holds.
To prove~\eqref{alpha_tilde_proof}, from~\eqref{neighborhood_A2} and the definition of $\nu^i_k$ we can write
\[
\nu^i_k \le \max_{j=1,\ldots,n} \tilde \alpha^j_k \le \bigl(\zeta(x^*)-L\norm{y^i_k-x^*}\bigr) \frac2{\Lmax}.
\]
Using~\eqref{zeta_proof} with $z = y^i_k$ and the fact that $L_i \le \Lmax$, we thus get~\eqref{alpha_tilde_proof}.

In view of~\eqref{alpha_tilde_proof} and Proposition~\ref{prop:d2}, it follows that
\begin{equation}\label{d_k}
d^i_k = -e_i,
\end{equation}
that is, $y^{i+1}_k = y^i_k - \alpha^i_k e_i$.
Using $z=y^i_k$ in~\eqref{neigh1}, we also have
\begin{equation}\label{alpha_max}
\alphamax = \abs{(y^i_k)_i-l_i} \le \frac{2 \nabla_i f(y^i_k)}{L_i + 2\gamma},
\end{equation}
where $\alphamax$ is the largest feasible stepsize along the direction $d^i_k$ at $y^i_k$.
So, if $\alphamax = 0$, then $\alpha_k^i = 0$ and, using~\eqref{d_k}, we have $(y_k^{i+1})_i =  (y_k^i)_i = l_i$, thus proving~\eqref{identification_proof1}.
If $\alphamax > 0$, using Proposition~\ref{prop:alpha_lips_descent},  \eqref{d_k} and~\eqref{alpha_max}, it follows that a sufficient decrease of $f$ along $d^i_k$ is  obtained with the first stepsize $\bar \alpha$ used in the line search, that is, the condition at line~4 of Algorithm~\ref{alg:ls} is satisfied.
Now, consider the extrapolation phase in the line search procedure, that is, lines~12--15 of Algorithm~\ref{alg:ls}.
Recalling~\eqref{d_k}, each stepsize $\omega = \min\{\alpha/\delta,\alphamax\}$ is such that $\alpha \le \omega \le (y^i_k)_i - l_i$, that is,
\[
0 \le {\omega} - \alpha \le (y^i_k + \alpha d^i_k)_i - l_i.
\]
So, from~\eqref{d_k} and the fact that $x^*_i = l_i$, it follows that
$\|y^i_k + \alpha d^i_k - x^*\| \le \|y^i_k-x^*\|$. So, we can apply~\eqref{neigh1} with $z = y^i_k + \alpha d^i_k$ and then we obtain
\[
0 \le {\omega} - \alpha \le (y^i_k + \alpha d^i_k)_i - l_i \le \frac{2 \nabla_i f(y^i_k + \alpha d^i_k)}{L_i + 2\gamma}
\]
for every stepsize $\omega$ used in the extrapolation.
Then, using Proposition~\ref{prop:alpha_lips_descent} with $x=y^i_k + \alpha d^i_k$, $s = \omega - \alpha$ and $d^i_k$ as in~\eqref{d_k}, it follows that
\[
f(y^i_k + {\omega} d^i_k) \le f(y^i_k + \alpha d^i_k) - \gamma ({\omega}- \alpha)^2.
\]
Namely, a sufficient decrease of $f$ is obtained with all stepsizes used in the extrapolation and we quit when we get the largest feasible stepsize, meaning that $(y^{i+1}_k)_i$ will be at the lower bound $l_i$ and thus proving~\eqref{identification_proof1}.
\end{proof}

Note that Theorem~\ref{main_identif_theorem} establishes finite identification for any limit point of $\{x_k\}$, thus not requiring the convergence of the whole sequence.
Note also, in the proof of Theorem~\ref{main_identif_theorem}, the crucial role played by the extrapolation in the line search procedure. Loosely speaking, when we are sufficiently close to a stationary point, expanding the stepsize allows us to hit the lower or the upper bound, provided the strict complementarity condition holds. This guarantees to identify all the variables satisfying the strict complementarity after a finite number of iterations.

Now, let us point out a useful property for the limit points of $\{x_k\}$. To this aim, let us define $X^\star$ as the set of all limit points of $\{x_k\}$, i.e.,
\[
X^\star := \{x \colon\exists K\subseteq\N \text{ such that } \lim_{k\to\infty, k\in K}x_k = x\}.
\]
The following result roughly states that,
if for all $x\in X^\star$ there does not exists $i$ such that $x_i$ violates the strict complementarity, then either all $x_i$ lie on the same bound or they all are strictly feasible.

\begin{proposition}\label{prop_utile1}
Let $\{x_k\}$ be the sequence of points produced by Algorithm~\ref{alg:dfl} and consider an index $i \in \{1,\ldots,n\}$.
Assume that there is no $x\in X^\star$ such that $i \in \Astar(x) \setminus \Astar^+(x)$.
Then,
\begin{itemize}
\item if there exists $x^* \in X^\star$ such that $i \in \Astar^+(x^\star)$, we have that $x_i = x^*_i$ for all $x\in X^\star$;
\item otherwise, $x_i \in (l_i, u_i)$ for all $x\in X^\star$.
\end{itemize}
\end{proposition}

\begin{proof}
We limit ourselves to show only the first point since the second one can be obtained as a logical consequence.
From Theorem~\ref{th:stationary} and recalling Ostrowski's theorem~\cite{bertsekasbook}, the set $X^\star$ of limit points of the sequence $\{x_k\}$ is a connected set. Now, let us consider any two points in $X^\star$, say $\bar x\in X^\star$ and $\bar y\in X^\star$, such that $\bar x\neq\bar y$,  and $\bar y_i \in \{l_i, u_i\}$.
Since $X^\star$ is connected, there exists a continuous function $\rho:[a,b]\to\R^n$ such that $\rho(a) = \bar x$, $\rho(b) = \bar y$ and $\rho(t)\in X^\star$, i.e., $\rho(t)$ {is stationary, for all} $t\in[a,b]$. Let us assume, without loss of generality, that $\bar y_i = l_i$ (the proof for the case $\bar y_i = u_i$ is identical, except for minor changes).
By contradiction, now assume that $\bar x_i > l_i$.
Since $\rho(a)_i = \bar x_i > l_i$ and $\rho(b)_i = \bar y_i = l_i$, then there exists $\bar t\in(a,b]$ such that $\rho(t)_i > l_i$ for all $t\in [a,\bar t)$ and $\rho(\bar t)_i = l_i$. Furthermore, by the stationarity conditions given in Definition~\ref{def:stat_grad}, we have that $\nabla_i f(\rho(t)) = 0$ for all $t\in[a,\bar t)$ and $\nabla_i f(\rho(\bar t)) > 0$, where the last inequality follows from the stated hypothesis. Then, by continuity of $\nabla f$, a scalar $\hat t\in(a,\bar t)$ must exist such that $\nabla_i f(\rho(t)) > 0$ for all $t\in (\hat t,\bar t]$.
This is a contradiction since $\nabla_i f(\rho(t)) = 0$ for all $t\in(\hat t,\bar t)$.
\end{proof}

Applying the above proposition for all indices $i \in \{1,\ldots,n\}$, the following result immediately follows, enforcing the finite active-set identification property established in Theorem~\ref{main_identif_theorem} when all the limit points of  $\{x_k\}$ are non-degenerate.

\begin{corollary}
Let $\{x_k\}$ be the sequence of points produced by Algorithm~\ref{alg:dfl} and assume that every $x\in X^\star$ is non-degenerate.
Then, for any pair $x',x'' \in X^\star$, we have $\Astar(x')=\Astar(x'')$ and $x'_i = x''_i$ for all $i \in \Astar(x')$. \end{corollary}

\section{Conclusions}\label{sec5}
In this paper, we have analyzed a derivative-free line search method for bound-constrained problems where the objective function has a Lipschitz continuous gradient.
For this algorithm, we have first provided complexity results. In more detail, given a prespecified threshold $\epsilon > 0$, we have shown that the criticality measure $\chi(x_k)$ (which vanishes at stationary points) falls below $\epsilon$ after at most $\bigO{ (n\epsilon^{-2})}$ iterations, requiring at most $\bigO{(n^2\epsilon^{-2})}$ function evaluations. These bounds match those obtained for (deterministic) direct-search~\cite{gratton2019direct} and model-based~\cite{hough2022model} methods.
Additionally, we have established an upper bound of $\bigO{(n^2\epsilon^{-2})}$ on the \textit{total} number of iterations where $\chi(x_k) \ge \epsilon$.
The latter result is obtained thanks to the extrapolation strategy used in the proposed line search, allowing us to upper bound $\chi(x_k)$ on both successful and unsuccessful iterations.

In the last part of the paper, we have considered the active-set identification property of the proposed method, i.e., the capability to detect the variables lying at the lower or the upper bound in the final solutions. In this respect, we have shown that, in a finite number of iterations, the algorithm identifies the active constraints satisfying the strict complementarity condition.
Also this property is obtained by exploiting the extrapolation used in the proposed line search, allowing the stepsize to expand, when we are in a neighborhood of a stationary point, until we hit the boundary of the feasible set.

Finally, some topics for future research can be envisaged. In particular, under convexity assumptions, the worst-case complexity of the algorithm might be tightened, in order to match the results given in~\cite{dodangeh2016worst},
and a bound on the maximum number of iterations required to identify the active constraints might be given.
We wish to report more results in future works.

\bibliography{df_active_set}
\end{document}